\documentclass[letterpaper, 11 pt, conference,onecolumn]{ieeeconf}
\usepackage{amsmath,amsfonts,amssymb}
\usepackage{graphicx,epsfig,psfrag,subfigure,remark}
\usepackage{xfrac,bm}
\usepackage{algorithmic,algorithm}
\usepackage[all]{xy}
\usepackage{varioref}
\usepackage{wrapfig}
\usepackage{threeparttable}
\usepackage{dcolumn}
\newcolumntype{d}{D{.}{.}{-1}}
\usepackage{nomencl}
\makeglossary
\usepackage{subfigure}
\usepackage{comment,cite}
\usepackage{xfrac}
\usepackage{mathrsfs}
\IEEEoverridecommandlockouts 

\newcommand{\x}{{\bm{x}}}
\newcommand{\X}{{\mathsf{X}}}

\newcommand{\f}{{\mathsf{f}}}

\newcommand{\z}{{\bm{z}}}

\newcommand{\cC}{\mathcal{C}}

\newcommand{\cE}{\mathcal{E}}

\newcommand{\cJ}{\mathcal{J}}
\newcommand{\cK}{\mathcal{K}}

\newcommand{\cM}{\mathcal{M}}
\newcommand{\cN}{\mathcal{N}}

\newcommand{\cV}{\mathcal{V}}

\newcommand{\mT}{T}

\newtheorem{proposition}{Proposition}

\newtheorem{problem}{Problem}
\newtheorem{theorem}{Theorem}

\renewcommand{\t}{^{\mbox{\scriptsize \mT}}}

\renewcommand{\baselinestretch}{1.20}
\newremark{remark}{Remark}

\title{\textbf{Minimum Variance and Covariance Steering Based on Affine Disturbance Feedback Control Parameterization}}

\author{Efstathios Bakolas \thanks{E. Bakolas
is an Associate Professor in the Department of Aerospace Engineering
and Engineering Mechanics, The University of Texas at Austin,
Austin, Texas 78712-1221, USA, Email: bakolas@austin.utexas.edu}}

\begin{document}

\maketitle

\begin{abstract}
The goal of this paper is to address finite-horizon minimum-variance and covariance steering problems for discrete-time stochastic (Gaussian) linear systems. On the one hand, the minimum variance problem seeks for a control policy that will steer the state mean of an uncertain system to a prescribed quantity while minimizing the trace of its terminal state covariance (or variance). On the other hand, the covariance steering problem seeks for a control policy that will steer the covariance of the terminal state to a prescribed positive definite matrix. We propose a solution approach that relies on the stochastic version of the affine disturbance feedback control parametrization according to which the control input at each stage can be expressed as an affine function of the history of disturbances that have acted upon the system. Our analysis reveals that this particular parametrization allows one to reduce the stochastic optimal control problems considered herein into tractable convex programs with essentially the same decision variables. This is in contrast with other control policy parametrizations, such as the state feedback parametrization, in which the decision variables of the convex program do not coincide with the controller's parameters of the stochastic optimal control problem. In addition, we propose a variation of the control parametrization which relies on truncated histories of past disturbances. We show that by selecting the length of the truncated sequences appropriately, we can design suboptimal controllers which can strike the desired balance between performance and computational cost. 
\end{abstract}

\section{Introduction}\label{s:intro}
One of the most fundamental problems in linear system theory is the
(finite-horizon) controllability problem which seeks for a control
signal (in continuous time) or a control sequence (in discrete time)
that will steer the system from a given initial state to a
prescribed one at a given final time. In the case of stochastic
linear systems, the controllability
problem can admit different interpretations and problem
formulations. In this work, we seek for control policies that will steer the  mean (first moment) and the covariance / variance (centered second moment) of the terminal state of a discrete-time linear stochastic system ``close'' to the respective goal quantities.

We will consider two practical variations of the latter stochastic control problem. In the first problem formulation, we seek for a control policy that will steer the mean of the terminal state to a prescribed vector and more importantly, the terminal state covariance / variance\footnote{The terms variance and covariance will be used interchangeably in order to abide to the standard terminology from the control literature; refer to Section \ref{sub:notation} for the precise definitions} to a prescribed positive definite matrix
while minimizing the control effort used for the latter transfer. We refer to the latter problem as the optimal \textit{covariance steering} problem. 
Due to the constraint on the terminal state covariance, the covariance steering problem does not correspond to a standard Linear Quadratic Gaussian (LQG) control problem. In many practical problems, however, it may not be clear how to choose a ``good'' terminal state covariance. More importantly in the presence of input constraints, the existence of admissible control inputs that will satisfy the hard terminal constraints on the state covariance may not be easily verifiable. In the second problem formulation, we seek a control policy that will steer the mean of the terminal state to a prescribed vector while minimizing the trace of the terminal state covariance / variance. The latter quantity can be viewed as a measure of the dispersion of the end points of sample trajectories of the system around the terminal mean state. Since there are no explicit specifications on the terminal state covariance, it is likely that the corresponding control policy will require the use of excessive control effort in order to reach the desired terminal mean state with the maximum ``accuracy.'' In order to avoid this, we introduce a constraint on the maximum expected value of the control effort that can be used by the controller. We will refer to the latter stochastic optimal control problem as the \textit{constrained minimum variance steering} problem. 

\textit{Literature Review:} Infinite-horizon stochastic control problems with constraints on the terminal state covariance in both
continuous-time and discrete-time settings have been addressed by Skelton and his
co-authors in \cite{p:skeltonIJC,p:skeltonTAC,p:yasuda1993,p:Grig97,p:levy1997discrete}.
The finite-horizon problem for the continuous-time case was recently addressed in
\cite{p:georgiou15A,p:georgiou15B,p:georgiou15C} whereas the discrete-time case in~\cite{p:bakcdc16,p:PT2017,p:BAKOLAS2018,p:golds2017,p:okamoto2018}. Covariance control problems for the partial information case can be found in \cite{p:EBACC17,p:bakCDC2018,p:bakTAC2019,p:ridderhof2020,p:kotsalis2020}. 
There exist also problem formulations which utilize soft constraints on the terminal state covariance in the form of appropriate terminal costs. Characteristic examples include the squared Wasserstein distance~\cite{p:halderacc16} or the squared $\mathcal{L}_2$ spatial norm~\cite{p:grune2019} between the goal distribution and the distribution attained by the terminal state.

In our previous work, we have addressed covariance steering and
minimum variance steering problems for discrete-time stochastic linear systems under both full state and partial state information based on convex optimization techniques~\cite{p:bakcdc16,p:EBACC17,p:EBSCL2018,p:BAKOLAS2018,p:bakTAC2019}.
In these references, the reduction of the stochastic optimal control problems to convex programs relied on the utilization of the so-called state feedback control parametrization~\cite{p:GOULART2006}. According to this parametrization, the control
input at each stage corresponds to an affine function of the states (or outputs) visited up to the present stage.
With the latter parametrization, the stochastic optimal control problem can be reduced to a convex program whose decision variables, however, do not coincide with the controller's parameters; the new decision variables are obtained from the controller's parameters by means of a bilinear transformation~\cite{p:skaf2010}.


%


\textit{Main Contribution:} In this paper, we present a new solution approach to the minimum variance and covariance steering 
problems under the assumption of full state information. Our approach is based on a control policy parameterization which can be interpreted as the stochastic version of the affine disturbance feedback control parametrization~\cite{p:bennemiro2004}. In the proposed control
policy parametrization, the control input at each stage can be
expressed as an affine function of the (complete or a truncated version of the) history of disturbances that have acted upon the system. This paper is a natural extension of \cite{p:bakCDC2018}, in which we employed a similar parametrization for the incomplete and imperfect state information case for the covariance steering problem only. In \cite{p:bakCDC2018}, the control input at each stage was taken to be an affine function of the history of output residuals computed via a Kalman filter. See also \cite{p:kotsalis2020} for an alternative parametrization based on the idea of purified output measurements.
By using this particular control policy
parametrization, one can directly reduce both the covariance and minimum variance control problems into tractable convex optimization problem, whose decision variables are essentially the same with the controller parameters. This is in sharp contrast with approaches that rely on the state feedback control parametrization
which require significant pre-processing in order to associate the controller parameters (decision variables of the stochastic optimal control problem) with the decisions variables of the corresponding convex program by means of bilinear transformations. The proposed policy parametrization does not require a similar pre-processing because the decision variables of the stochastic optimal control problems and the corresponding convex programs are directly associated with each other (no use of a bilinear transformation is needed). In addition, one can consider a variation of the proposed control policy parametrization in which the control input relies on truncated histories of past disturbances. By appropriately selecting the length of the truncated histories of the past disturbances, one can tune the size of the convex programs which will yield suboptimal controllers which strike the desired balance between performance and computational cost.

\textit{Structure of the paper:} The rest of the paper is organized
as follows. In Section~\ref{s:form}, we formulate the minimum variance and covariance steering problems. The control policy parameterization and the reduction of the two problems into equivalent convex programs is described in Section~\ref{s:convex}. 
%
Finally, Section~\ref{s:concl} concludes
the paper with a summary of remarks.

\section{Problem Formulation}\label{s:form}

\subsection{Notation}\label{sub:notation}

We denote by $\mathbb{R}^n$ the set of $n$-dimensional real vectors
and by $\mathbb{Z}$ the set of integers. We write $\mathbb{E}\left[
\cdot\right]$ to denote the expectation operator. Given two integers
$\tau_1, \tau_2$ with $\tau_1\leq \tau_2$, 
then $[\tau_1, \tau_2]_d := [\tau_1,\tau_2] \cap \mathbb{Z}$. Given a sequence $\mathscr{X}:=\{x_i:~i\in[1,m]_d\}$, 
we denote by $\mathrm{vertcat}(\mathscr{X})$ the concatenation of the $m$ vectors of $\mathscr{X}$, that is, $\mathrm{vertcat}(\mathscr{X}) := [x_1\t, \dots, x_m\t]\t$. Given a
square matrix $\mathbf{A}$, we denote its trace by
$\mathrm{trace}(\mathbf{A})$. We write $\mathbf{0}$ and $I_n$ to
denote the zero matrix and the $n$-dimensional identity matrix, respectively. The space of real symmetric $n\times n$ matrices will be denoted by
$\mathbb{S}_n$. Furthermore, we will denote the convex cone of
$n\times n$ (symmetric) positive semi-definite and (symmetric)
positive definite matrices by $\mathbb{S}^{+}_n$ and
$\mathbb{S}^{++}_n$, respectively. Given two matrices $A, B \in
\mathbb{S}^{+}_n$, then we write $B \preceq A$ if and only if $A-B \in \mathbb{S}^{+}_n$, where $\preceq$ denotes the partial Loewner order in $\mathbb{S}^{+}_n$. Given $A \in \mathbb{S}^+_n$, we denote by $A^{1/2} \in \mathbb{S}^{+}_n$ its (unique) square root, that is, $A^{1/2} A^{1/2} = A$. Let $x\in \mathbb{R}^n$, $\mathbf{\Sigma}  \in \mathbb{S}^{++}_n$, and $\gamma > 0$, we write $\cE_{\gamma}(x;\mathbf{\Sigma})$ to denote the ellipsoid $\{z\in\mathbb{R}^n: (z - x)\t \mathbf{\Sigma}^{-1} (z -x) \leq \gamma\}$.

Finally, we write $\mathrm{bdiag}(A_1,$ $\dots, A_\ell)$ to denote the block diagonal matrix formed by the matrices $A_i$, $i\in \{1,\dots, \ell\}$. We denote the mean and the covariance / variance of a random vector $z$ by, respectively, $\mu_z$ and $\mathrm{var}_z$, where $\mu_z:=\mathbb{E}[z]$ and $\mathrm{var}_z:= \mathbb{E}[(z-\mu_z)(z-\mu_z)\t]=\mathbb{E}[zz\t]-\mu_z \mu_z\t$. Given a discrete stochastic process $\{x(t;\omega):t \in [0,T]_d\}$, or more compactly $\{x(t):t \in [0,T]_d\}$, defined on a probability space $(\Omega,\mathfrak{F}, \mathbb{P})$, we denote by $\mathrm{cov}_x(t_1, t_2)$, for $t_1$, $t_2\in[0,T]_d$, the covariance (or auto-covariance) of $x(\cdot)$, where $\mathrm{cov}_x(t_1, t_2) := \mathbb{E}\left[
(x(t_1)- \mu_x(t_1))(x(t_2) - \mu_x(t_2))\t\right]$ and $\mu_x(t) = \mathbb{E}\left[ x(t) \right]$. We denote by $\mathrm{var}_x(t)$, the variance of $x(t)$, which is defined as $\mathrm{var}_x(t) := \mathrm{cov}_x(t, t)$, for $t \in [0,T]_d$. Previous definitions are adopted from the classic book on random variables and stochastic processes~\cite{b:papoulis2002}.

\section{Preliminaries}\label{s:prelim}
\subsection{Problem setup}

We consider the following discrete-time stochastic linear system
\begin{subequations}
\begin{align}\label{eq:motion}
x(t+1) & = A(t) x(t) + B(t) u(t) + w(t), \\
x(0) & =x_0,~~ \quad ~~x_0 \sim \cN(\mu_0, \Sigma_0), \label{eq:motion2}
\end{align}
\end{subequations}
for $t \in [0,T-1]_{d}$, where $\mu_0\in
\mathbb{R}^n$ and $\Sigma_0 \in \mathbb{S}_n^{++}$ are given. We
denote by $\mathscr{X}_{0:t} := \{ x(\tau) \in \mathbb{R}^n: \tau \in [0,t]_{d} \}$, for $t\in [0,T]_d$, the state (random) process 
and by $\mathscr{U}_{0:t} := \{ u(\tau) \in
\mathbb{R}^m:~\tau\in [0, t]_{d}\}$, for $t\in [0,T-1]_d$, the input process. Finally, we
denote by $\mathscr{W}_{0:t} := \{ w(\tau) \in \mathbb{R}^n:~\tau\in
[0,t]_{d}\}$, for $t\in [0,T-1]_d$, the noise process, which is assumed to be a sequence of independent and identically distributed normal random variables
with
\begin{align}\label{eq:w1}
\mu_w(t) & = \mathbf{0}, \qquad \mathrm{cov}_w(t, \tau) = \delta(t,\tau) W,
\end{align}
for all $t,\tau\in [0, T-1]_{d}$, where $W\in\mathbb{S}_n^{+}$  and
$\delta(t,\tau):=1$, when $t=\tau$, and $\delta(t,\tau):=0$,
otherwise. It is worth noting that by assuming that $W \in \mathbb{S}_n^{+}$ instead of $W \in \mathbb{S}_n^{++}$ , we cover the case in which $w(t) = D v(t)$ where $D \in \mathbb{R}^{n \times p}$ and $\mathrm{cov}_v(t, \tau) = \delta(t,\tau) V$ with $V \in\mathbb{S}_n^{++}$. In this case, $W = D V D\t \in \mathbb{S}_n^+$ and thus there is no loss of generality in considering a state space model where the noise vector has the same dimension with the state vector as in \eqref{eq:motion}. Furthermore, $x_0$ is independent of $W_{0: T-1}$, that is,
\begin{align}\label{eq:x0w}
\mathbb{E}\left[ x_0  w(t)\t  \right] & =
\mathbf{0}, \qquad \mathbb{E}\left[ w(t) x_0\t  \right] = \mathbf{0},
\end{align}
for all $t \in [0, T-1]_d$. All random variables are defined on a (fixed) complete probability space $(\Omega, \mathfrak{F}, \mathbb{P})$. Finally, we assume perfect state information, that is, at each time $t$ the system is aware of the exact realization of the state process at the same time as well as all previous stages (by keeping track of all the states visited by
the system). More precisely, if a particular $\omega\in\Omega$ corresponds to an experimental outcome, then the information available to the system at time $t$ is the sequence $\{ x(\tau; \omega):~\tau\in [0,t] \}$ (sample trajectory of the state process $\mathscr{X}$ corresponding to the particular $\omega\in\Omega$). Throughout the paper, we will drop the dependence of random variables on the experimental outcome to keep the notation simple.

Equation~\eqref{eq:motion} can be written more compactly as follows:
\begin{align}\label{eq:bmx}
\bm{x} = \mathbf{G}_{\bm{u}} \bm{u} + \mathbf{G}_{\bm{w}} \bm{w} +
\mathbf{G}_{0} x_0,
\end{align}
where $\bm{x} := \mathrm{vertcat}(\mathscr{X}_{0:T})  \in \mathbb{R}^{(T+1)n}$, $\bm{u} := \mathrm{vertcat}(\mathscr{U}_{0:T-1}) \in\mathbb{R}^{Tm}$
and $\bm{w} := \mathrm{vertcat}(\mathscr{W}_{0:T-1}) \in\mathbb{R}^{Tn}$.
Furthermore,
\begin{align*}
& \mathbf{G}_{\bm{u}} := \left[ \begin{smallmatrix} \mathbf{0} & \mathbf{0} & \dots & \mathbf{0} \\
B(0) &  \mathbf{0} & \dots & \mathbf{0} \\
\Phi(2,1)B(0) & B(1) & \dots & \mathbf{0} \\
\vdots & \vdots & \dots & \vdots \\
\Phi(T,1)B(0) & \Phi(T,2)B(1) &
\dots & B(T-1)\end{smallmatrix} \right],\\
&\mathbf{G}_{\bm{w}} := \left[ \begin{smallmatrix} \mathbf{0} & \mathbf{0} & \dots & \mathbf{0} \\
I & \mathbf{0} & \dots & \mathbf{0} \\
\Phi(2,1) & I & \dots & \mathbf{0} \\
\vdots & \vdots & \dots & \vdots \\
\Phi(T,1) & \Phi(T,2) & \dots & I
\end{smallmatrix} \right],~~\mathbf{G}_{0}:=
\left[ \begin{smallmatrix} I \\ \Phi(1,0) \\ \vdots \\ \Phi(T,0)
\end{smallmatrix}\right],
\end{align*}
where $\Phi(t,\tau) := A(t-1)\dots A(\tau)$ and $\Phi(t,t) = I$,
for $t\in [1,T]_d$ and $\tau \in [0,t-1]_d$ with $t>\tau$.

Furthermore, the first two moments of $\bm{w}$ are given by
\begin{equation}\label{eq:bmw}
\mathbb{E}\big[ \bm{w} \big] = \bm{0},~~~\mathbb{E}\big[ \bm{w} \bm{w}\t\big] = \mathbf{W},
\end{equation}
where in view of \eqref{eq:w1}
\begin{align*}
 \mathbf{W} & := \mathbb{E}\big[
\mathrm{bdiag}(w(0)w(0)\t, \dots, w(T-1) w(T-1)\t) \\
& = \mathrm{bdiag}(W, \dots, W).
\end{align*}


\subsection{Controller Parametrization and Problem Formulation}

As we have already mentioned, the assumption of perfect state
information implies that at each stage $t\in[1,T]_d$ the realization of the state process up to and including stage $t$, $\mathscr{X}_{0:t}$, is perfectly known to the system. Similarly, the (deterministic) sequence $\hat{\mathscr{X}}_{0:t} :=\{ \hat{x}(0),
\dots, \hat{x}(t)\}$ where $\hat{x}(0) = x(0)$ and $\hat{x}(\tau)$, for $\tau \in [1,t]_d$, corresponds to the state that the system would attain if no stochastic disturbances were acting upon the system at $t=\tau-1$  and at the same stage the system was at state $x(\tau-1)$ (the state attained at the same stage by the stochastic system which was subject to disturbances in the previous stages) and the input $u(\tau-1)$ was applied to it (the same input was also applied to the stochastic system at the same stage); we can write $x(t) = x(\tau =t | x(\tau-1),u(\tau-1), w(\tau-1))$ and $\hat{x}(t) = x(\tau =t | x(\tau-1),u(\tau-1), 0)$. 
Clearly,
\begin{align}\label{eq:error}
x(t) - \hat{x}(t) & = x(\tau =t | x(\tau-1),u(\tau-1), w(\tau-1)) \nonumber \\
&~~~ - x(\tau = t | x(\tau-1),u(\tau-1), 0) \nonumber \\
& = w(t - 1),
\end{align}
and given that both $x(t)$ and $\hat{x}(t)$ are perfectly known at stage $\tau = t$, we conclude that the disturbance at the previous stage, $w(t - 1)$, is also
perfectly known. More precisely, the whole past history of the particular
realization of the noise process $W_{0:t-1} := \{ w(\tau):~\tau \in
[0, t - 1] \}$ is known to the system at stage $t$. Therefore,
it is possible, based on the information available, to utilize control policies which are sequences of control laws that are affine functions of the elements of (any specific realization of)
$W_{0:t-1}$, that is,
\begin{align}\label{eq:policy}
\kappa(t,W_{0:t-1})& = \begin{cases} \bar{u}(t) + \sum_{\tau =0}^{t-1}
K_{(t-1, \tau)} w(\tau),\\
~~\qquad~~~\qquad~~\mathrm{if}~ t \in[1,T-1]_d,\\
\bar{u}(0),~~~\qquad~~~\mathrm{if}~ t=0,
\end{cases}
\end{align}
where $K(t-1,\tau)\in \mathbb{R}^{m\times n}$ for all $\tau\in[0,t-1]_d$.
The above control policy parametrization corresponds to the
stochastic version of the so-called affine disturbance feedback
parametrization~\cite{p:bennemiro2004}.

If we set $u(t)=\kappa(t,W_{0:t-1})$, for all $t\in[0, T-1]_d$ with 
$\kappa(t,W_{0:t-1})$ as defined in \eqref{eq:policy}, then the
resulting closed-loop dynamics can be written as follows:
\begin{align}\label{eq:cloop}
x(t+1) & = A(t) x(t) + B(t)\bar{u}(t) \nonumber \\
&~~~ + B(t) \Big( \sum_{\tau =0}^{t-1} K_{(t-1,\tau)} w(\tau) \Big) +
w(t).
\end{align}

Next, we provide the precise formulations of the minimum variance steering and covariance steering problems based on the control policy parametrization which is described in \eqref{eq:policy}.

\begin{problem}[Minimum Variance Steering]\label{problem1}
Let $\mu_\f \in \mathbb{R}^n$ be given. Consider the system described by~\eqref{eq:cloop}. Then, find the
collection of matrix gains $\mathscr{K} := \{ K(t,\tau): (t,\tau) \in
[0,T-2]\times [0,T-2],~t \geq \tau \}$
and the sequence of vectors $\overline{\mathscr{U}} := \{\bar{u}(0),
\dots, \bar{u}(T-1) \}$ that minimize the following performance
index:
\begin{align}\label{eq:perfindex}
J_1(\overline{\mathscr{U}}, \mathscr{K}) & := \mathrm{tr}(\mathrm{var}_x(T) ),
\end{align}
subject to the input constraint $C(\overline{\mathscr{U}}, \mathscr{K}) \leq 0$, where
\begin{equation}\label{eq:inputeffort}
C(\overline{\mathscr{U}}, \mathscr{K}) := \mathbb{E}\Big[  \sum_{t=0}^{T-1}
u(t)\t u(t) \Big] - \rho^2,
\end{equation}
as well as the boundary condition on the terminal mean
\begin{equation}\label{eq:termmean1}
\mu_x(T) = \mu_\f. 
\end{equation}
\end{problem}

\begin{remark}
The performance index $J_1(\overline{\mathscr{U}}, \mathscr{K})$ defined in
\eqref{eq:perfindex} corresponds to a single terminal cost term. The input constraint \eqref{eq:inputeffort} is imposed in order to avoid using excessive control effort to achieve unnecessarily ``small'' terminal variance or ``accuracy''. In particular, the parameter $\rho$ will have to be tuned so that it strikes the desired balance between accuracy (measured in terms of $\mathrm{tr}(\mathrm{var}_x(T) )$) and the cost incurred to achieve the latter accuracy (measured in terms of the control effort over the time horizon $[0,T-1]_d$).
\end{remark}

\begin{problem}[Covariance Steering]\label{problem2}
Let $\mu_\f \in \mathbb{R}^n$ and $\Sigma_\f \in \mathbb{S}^{++}_n$ be given. Consider the system described by~\eqref{eq:cloop}. Then, find the
collection of matrix gains $\mathscr{K} := \{ K_{(t,\tau)}: (t,\tau) \in
[0,T-2]\times [0,T-2],~t \geq \tau \}$
and the sequence of vectors $\overline{\mathscr{U}} := \{\bar{u}(0),
\dots, \bar{u}(T-1) \}$ that minimize the performance
index
\begin{align}\label{eq:perfindex2}
J_2(\overline{\mathscr{U}}, \mathscr{K}) & := 
\sum_{t=0}^{T-1}\mathbb{E}[ u(t)\t u(t) ]
\end{align}
subject to the following boundary conditions on the terminal state mean and covariance:
\begin{equation}\label{eq:termcov}
\mu_x(T) = \mu_\f,~~~~(\Sigma_\f - \mathrm{var}_x(T))
\in \mathbb{S}_{n}^{+}.
\end{equation}
\end{problem}

\begin{remark}
The positive semi-definite constraint on the terminal state variance $(\Sigma_\f - \mathrm{var}_x(T)) \in \mathbb{S}_{n}^{+}$ corresponds to a relaxation of the hard inequality constraint $\mathrm{var}_x(T)
= \Sigma_\f$, which determines, however, a set of non-convex equality constraints~\cite{p:BAKOLAS2018}. 
The constraint $(\Sigma_\f - \mathrm{var}_x(T) )\in \mathbb{S}_{n}^{+}$ places a practical upper bound on the minimum acceptable accuracy (measured in terms of the second centered moment) with which the system's mean state will reach the goal vector $\mu_\f$.
\end{remark}

\begin{remark}
The performance index $J_2(\overline{\mathscr{U}}, \mathscr{K})$ defined in
\eqref{eq:perfindex2} corresponds to the expected value of the control effort over $[0, T-1]_d$ which is required for the transfer of the state mean and covariance of the stochastic system~\eqref{eq:motion} to their desired terminal quantities. It is worth noting that in the formulation of Problem~\ref{problem2}, the requirement on the accuracy at which the state reaches the desired terminal mean $\mu_\f$ is enforced in an explicit way by means of the constraint on the state covariance $\mathrm{var}_x(T)) \preceq \Sigma_\f$ in contrast with Problem~\ref{problem1}, in which the terminal cost, $\mathrm{tr}(\mathrm{var}_x(t))$, can be viewed as a ``soft'' (or indirect) constraint on the terminal accuracy. On the other hand, no explicit input constraints are imposed on the control effort that the system can use to achieve the desired accuracy in Problem~\ref{problem2} which is in contrast with Problem~\ref{problem1}, in which an explicit bound on the expected value of the control effort that can be used by the system is considered.
\end{remark}

\subsection{Reduction of Minimum Variance and Covariance Steering Problems to Tractable Convex Programs}

Next, we will reduce the minimum variance steering problem (Problem~\ref{problem1}) and the covariance steering problem (Problem~\ref{problem2}) to tractable convex programs. 
To this aim, we will obtain expressions for the mean and variance of the state of the closed-loop system that results by setting $u(t) = \kappa(t,W_{0:t-1})$, for all $t\in[0, T-1]$,
where $\kappa(t,W_{0:t-1})$ is defined in \eqref{eq:policy}. In particular, it follows readily that
\begin{equation}\label{eq:controlcmp}
\bm{u} = \bar{\bm{u}} + \bm{\mathcal{K}} \bm{w},
\end{equation}
where $\bar{\bm{u}} := \mathrm{vertcat}(\overline{\mathscr{U}})$ with 
\begin{equation}\label{eq:vertcat}
\mathrm{vertcat}(\overline{\mathscr{U}}:= \begin{bmatrix} \bar{u}(0)\t , \dots, 
\bar{u}(T-1)\t \end{bmatrix} \t
\end{equation}
and $\bm{\mathcal{K}} := \big[ \begin{smallmatrix}\mathbf{0} & \mathbf{0} \\ \mathbf{K} & \mathbf{0}\end{smallmatrix} \big]$ where $\mathbf{K}$ is the block lower-triangular matrix formed by the elements of $\mathscr{K}$; we write $\mathbf{K} := \mathrm{bltr}(\mathscr{K})$, where
\begin{align}\label{eq:mathcalK}
&  \mathrm{bltr}(\mathscr{K}) := \left[\begin{smallmatrix} 
K_{(0, 0)} & \mathbf{0}  &  \dots & \mathbf{0}  \\
K_{(1, 0)} & K_{(1, 1)}  &  \dots & \mathbf{0}  \\
\vdots & \vdots & \ddots & \vdots \\
K_{(T-2, 0)} & K_{(T-2,\sigma + 1)} & \dots & K_{(T-2, T-2)}
\end{smallmatrix} \right].
\end{align} 
The operator $\mathscr{K} \mapsto \mathrm{bltr}(\mathscr{K})$ returns a block lower-triangular matrix whose structure is determined by \eqref{eq:mathcalK} given a double sequence of matrices. We will denote by $\mathscr{D}$ the set of all pairs $(\bar{\bm{u}}, \bm{\mathcal{K}})\in \mathbb{R}^{Tm} \times \mathbb{R}^{Tm \times Tn}$ where $\bm{\mathcal{K}}$ is a block lower-triangular matrix whose precise structure we just described. We will see later on that $\mathscr{D}$ will correspond to the decision space of two convex programs which are equivalent to Problem 1 and Problem 2.

After plugging \eqref{eq:controlcmp} into \eqref{eq:bmx}, the closed-loop dynamics given in~\eqref{eq:cloop} can be written compactly as
follows:
\begin{align}\label{eq:bmxK}
\bm{x} &  = 
\mathbf{G}_{\bm{u}} \bar{\bm{u}} + (\mathbf{G}_{\bm{w}} +
\mathbf{G}_{\bm{u}} \bm{\mathcal{K}} ) \bm{w} + \mathbf{G}_{0} x_0.
\end{align}
Therefore, $x(t)$ can be extracted by $\bm{x}$ as follows:
\begin{equation}\label{eq:stateproj}
x(t) = \mathbf{P}_{t+1} \bm{x},~~~\forall t\in[0,T]_{d},
\end{equation}
where $\mathbf{P}_{t+1}$ is a block
matrix whose blocks are equal to the zero matrix except from the $(t+1)$-th
block which is equal to the identity matrix.

Next, we provide analytic expressions for the mean and variance of
$\bm{x}$ as well as the state $x(t)$ for $t\in [0,T]_{d}$. 
%
%
\begin{proposition}
The mean and variance of the random vector $\bm{x}$, which satisfies Eq.~\eqref{eq:bmxK}, are given by
\begin{equation}\label{eq:basicfg}
\mu_{\bm{x}} =
\mathfrak{f}(\bar{\bm{u}}),~~~~\mathrm{var}_{\bm{x}}
= \mathfrak{h}( \bm{\mathcal{K}}),
\end{equation}
where
\begin{subequations}
\begin{align}\label{eq:fu}
\mathfrak{f}(\bar{\bm{u}}) & := \mathbf{G}_{\bm{u}}
\bar{\bm{u}} + \mathbf{G}_{0}\mu_0, \\
\mathfrak{h}( \bm{\mathcal{K}}) & :=
 \mathbf{G}_{0} \Sigma_0 \mathbf{G}_{0}\t \nonumber \\
 & ~~~~ + (\mathbf{G}_{\bm{w}} + \mathbf{G}_{\bm{u}} \bm{\mathcal{K}} ) \mathbf{W}(\mathbf{G}_{\bm{w}} +
\mathbf{G}_{\bm{u}} \bm{\mathcal{K}} )\t.\label{eq:huK}
\end{align}
\end{subequations}
Furthermore, the mean and the variance of the state $x(t)$
satisfy, respectively, the following equations:
\begin{align}\label{eq:muvarx}
\mu_x(t)  = 
\mathbf{P}_{t+1}\mathfrak{f}(\bar{\bm{u}}),~~~  
\mathrm{var}_x(t)  = \mathbf{P}_{t+1} \mathfrak{h}(\bm{\mathcal{K}}) \mathbf{P}_{t+1}\t, 
\end{align}
for all $t\in[0,T]_{d}$. 
\end{proposition}
\begin{proof}
By applying the expectation operator at both sides of Equation \eqref{eq:bmxK}, we obtain
\begin{align}\label{eq:fproof}
\mu_{\bm{x}} := \mathbb{E}\left[ \bm{x} \right] & = \mathbf{G}_{\bm{u}} \bar{\bm{u}} +
\mathbf{G}_{0}\mu_0 =:\mathfrak{f}(\bar{\bm{u}}),
\end{align}
where in our derivation we have used \eqref{eq:motion2}, which implies that $\mathbb{E}[x_0] = \mu_0$, and \eqref{eq:w1}, which implies that $\mathbb{E}[\bm{w}] =
\bm{0}$. This proves the validity of Eq.~\eqref{eq:fu}. In addition, \eqref{eq:stateproj} and \eqref{eq:fproof} imply the first equation in~\eqref{eq:muvarx}.
Furthermore, in view of \eqref{eq:bmxK}, we have
\begin{align}\label{eq:gproof}
\mathrm{var}_{\bm{x}} & = \mathbb{E}\big[ \bm{x} \bm{x}\t \big] - \mu_{\bm{x}} \mu_{\bm{x}}\t \nonumber \\
 & = \mathbf{G}_{\bm{u}}
\bar{\bm{u}} \bar{\bm{u}}\t \mathbf{G}_{\bm{u}}\t  + \mathbf{G}_{0}
\mathbb{E}\big[ x_0 x_0\t \big]
\mathbf{G}_{0}\t \nonumber \\
&~~~ +  \mathbf{G}_{\bm{u}} \bar{\bm{u}} \mu_0\t \mathbf{G}_{0}\t  +
\mathbf{G}_{0} \mu_0 \bar{\bm{u}}\t
\mathbf{G}_{\bm{u}}\t \nonumber \\
&~~~ + \mathbf{G}_{\bm{w}} \mathbb{E}\big[ \bm{w} \bm{w}\t \big]
\mathbf{G}_{\bm{w}}\t + \mathbf{G}_{\bm{u}} \bm{\mathcal{K}}
\mathbb{E}\big[ \bm{w} \bm{w}\t \big] \bm{\mathcal{K}}\t
\mathbf{G}_{\bm{u}}\t
\nonumber \\
&~~~ +\mathbf{G}_{\bm{w}} \mathbb{E}\big[ \bm{w} \bm{w}\t \big]
\bm{\mathcal{K}}\t \mathbf{G}_{\bm{u}}\t \nonumber \\
&~~~+ \mathbf{G}_{\bm{u}}
\bm{\mathcal{K}} \mathbb{E}\big[ \bm{w} \bm{w}\t \big]
\mathbf{G}_{\bm{w}}\t - \mu_{\bm{x}} \mu_{\bm{x}}\t.
\end{align}
In view of \eqref{eq:bmw} and the identities $\mathbb{E}\big[ x_0 x_0\t\big] = \Sigma_0 +
\mu_0 \mu_0\t$ and $\mathbb{E}\big[ \bar{\bm{u}}
x_0\t\big]=\bar{\bm{u}} \mu_0\t$, it follows that 
\begin{align*}
\mathrm{var}_{\bm{x}} & = \mathbf{G}_{0} \Sigma_0 \mathbf{G}_{0}\t + 
(\mathbf{G}_{\bm{u}}
\bar{\bm{u}} + \mathbf{G}_{0}\mu_0)(\mathbf{G}_{\bm{u}}
\bar{\bm{u}} + \mathbf{G}_{0}\mu_0)\t + \nonumber \\
&~~~ + (\mathbf{G}_{\bm{w}} + \mathbf{G}_{\bm{u}} \bm{\mathcal{K}} ) \mathbf{W}(\mathbf{G}_{\bm{w}} +
\mathbf{G}_{\bm{u}} \bm{\mathcal{K}} )\t - \mu_{\bm{x}} \mu_{\bm{x}}\t,
\end{align*}
which in view of \eqref{eq:fproof} implies that 
\begin{align}\label{eq:gproof2}
\mathrm{var}_{\bm{x}} & = \mathbf{G}_{0} \Sigma_0 \mathbf{G}_{0}\t \nonumber \\
&~~~~+ (\mathbf{G}_{\bm{w}} \mathbf{G}_{\bm{u}} \bm{\mathcal{K}} ) \mathbf{W}(\mathbf{G}_{\bm{w}} +
\mathbf{G}_{\bm{u}} \bm{\mathcal{K}} )\t =: \mathfrak{h}( \bm{\mathcal{K}}).
\end{align}
Finally, \eqref{eq:stateproj} and \eqref{eq:gproof2} imply that
\begin{align*}
\mathrm{var}_x(t) & = \mathbb{E}\big[ x(t)  x(t)\t \big] - \mu_x(t) \mu_x(t)\t \\
& = \mathbf{P}_{t+1} \left(\mathbb{E}\big[ \bm{x} \bm{x}\t \big] - \mu_{\bm{x}} \mu_{\bm{x}}\t \right)
\mathbf{P}_{t+1}\t \\
& = \mathbf{P}_{t+1} \mathrm{var}_{\bm{x}} \mathbf{P}_{t+1}\t = \mathbf{P}_{t+1} \mathfrak{h}( \bm{\mathcal{K}}) \mathbf{P}_{t+1}\t,
\end{align*}
which proves the validity of \eqref{eq:muvarx}.
\end{proof}

Next, we obtain an expression for the performance index of the minimum variance steering problem (Problem~\ref{problem1}) in terms of the decision variables $(\bar{\bm{u}}, \bm{\mathcal{K}})$ and subsequently, we prove its convexity.

\begin{proposition}\label{prop:J1}
The performance index $J_1(\overline{\mathscr{U}}, \mathscr{K})$ which is defined in \eqref{eq:perfindex} is equal to 
$\cJ_1(\bm{\cK})$, where
\begin{align}\label{eq:J1uK}
\cJ_1(\bm{\cK}) & := \mathrm{tr}\big( \mathbf{P}_{T+1} \mathfrak{h}(\bm{\mathcal{K}}) \mathbf{P}_{T+1}\t \big),
\end{align}
%
provided that the pairs of decision variables $(\overline{\mathscr{U}},\mathscr{K})$ and $(\bar{\bm{u}},\bm{\mathcal{K}})$ are related by \eqref{eq:mathcalK}. Furthermore, $\cJ_1(\bm{\cK})$ is a convex function. 
\end{proposition}
\begin{proof}
%
%
In view of \eqref{eq:perfindex}, we have
\begin{align*}
J_1(\overline{\mathscr{U}}, \mathscr{K} ) & = \mathrm{tr} \big( \mathrm{var}_x(T) \big) \nonumber \\
& = \mathrm{tr}\big( \mathbf{P}_{T+1} \mathrm{var}_{\bm{x}} \mathbf{P}_{T+1}\t \big) \nonumber \\
& = \mathrm{tr}\big( \mathbf{P}_{T+1} \mathfrak{h}(\bm{\mathcal{K}}) \mathbf{P}_{T+1}\t \big) =: \cJ_1(\bm{\cK}),
\end{align*}
where in the last equality, we have used \eqref{eq:basicfg}.
The validity of \eqref{eq:J1uK} follows in light of \eqref{eq:huK}.
Furthermore, in view of \eqref{eq:huK}, we conclude that the performance index $J_1(\overline{\mathscr{U}}, \mathscr{K})$ is a convex function.
\end{proof}

\begin{proposition}\label{prop:constr}
Let $\rho>0$. Then, the input constraint given in \eqref{eq:inputeffort} and the terminal constraint given in \eqref{eq:termmean1} determine a convex subset $\mathscr{D}_{\mathrm{MVS}}$ of the decision space $\mathscr{D}$. 
\end{proposition}
\begin{proof}
The input constraint given in \eqref{eq:inputeffort} can be written as 
$C(\overline{\mathscr{U}}, \mathscr{K}) \leq 0$, where $C(\overline{\mathscr{U}}, \mathscr{K}):=\mathbb{E}[\sum_{t=0}^{T-1}  u(t)\t u(t)] - \rho^2$. It follows that 
\begin{align*}
C(\overline{\mathscr{U}}, \mathscr{K}) & = \mathbb{E}[ \mathrm{trace}( \bm{u}\bm{u}\t )] - \rho^2 \\
& = \mathrm{trace}(\mathbb{E}[ (\bar{\bm{u}} + \bm{\cK} \bm{w})
(\bar{\bm{u}} + \bm{\cK} \bm{w})\t ]) - \rho^2 \\
& = \bar{\bm{u}}\t \bar{\bm{u}} + \mathrm{trace}(\bm{\cK} \mathbf{W} \bm{\cK}\t)  - \rho^2 =: \cC(\bar{\bm{u}}, \bm{\mathcal{K}}),
\end{align*}
where in the last derivation we have used \eqref{eq:bmw}. Clearly, $\cC$ is a convex (quadratic) function and thus, the set $\mathscr{D}_1 := \{(\bar{\bm{u}}, \bm{\mathcal{K}}) \in \mathscr{D}:~ \cC(\bar{\bm{u}}, \bm{\mathcal{K}}) \leq \rho^2 \}$ is convex.

In addition, the equality constraint given in \eqref{eq:termmean1} can be written as follows: $M(\overline{\mathscr{U}}) = 0$, where $M(\overline{\mathscr{U}}) := \mu_x(T) - \mu_\f$. In view of \eqref{eq:muvarx}, we have
\begin{equation}\label{eq:meancons2}
M(\overline{\mathscr{U}}) = \mathbf{P}_{T+1} \mathfrak{f}(\bar{\bm{u}}) - \mu_\f =: \cM(\bar{\bm{u}}).
\end{equation}
Because $\cM(\bar{\bm{u}})$ is an affine function, the equality constraint $\cM(\bar{\bm{u}})=0$ determines a convex set $\mathscr{D}_2 \subseteq \mathscr{D}$.
\end{proof}

The following theorem is a direct consequence of Propositions~\ref{prop:J1} and \ref{prop:constr}.
\begin{theorem}\label{theorem1}
Problem~\ref{problem1} is equivalent to the following convex program:
\begin{equation}
\min \cJ_1(\bm{\mathcal{K}})~~\text{subject~to}~~(\bar{\bm{u}}, \bm{\mathcal{K}})\in \mathscr{D}_1 \cap \mathscr{D}_2,
\end{equation}
where $\mathscr{D}_1:=\{ (\bar{\bm{u}}, \bm{\mathcal{K}}) \in \mathscr{D}:~\cC(\bar{\bm{u}}, \bm{\mathcal{K}}) \leq 0 \}$ and $\mathscr{D}_2 = \{ (\bar{\bm{u}}, \bm{\mathcal{K}})  \in \mathscr{D}: \cM(\bar{\bm{u}}) =0 \}$.
\end{theorem}

\begin{remark}
Theorem~\ref{theorem1} implies that Problem~\ref{problem1} is equivalent to a convex quadratically constrained quadratic program (QCQP), that is, a convex quadratic problem subject to (convex) quadratic inequality and affine equality constraints.
\end{remark}

\section{Reduction of the covariance steering problem into a semi-definite program}\label{s:convex}
Next, we associate the covariance steering problem into a semi-definite program.

\begin{proposition}
The constraint on the terminal state covariance $\mathrm{var}_{x}(T) \preceq \Sigma_\f$, for a given matrix $\Sigma_\f \in \mathbb{S}^{++}_n$ and the terminal constraint on the terminal mean $\mu_{x}(T) = \mu_\f$, for a given vector $\mu_\f \in \mathbb{R}^n$, determine a convex subset $\mathscr{D}_{\mathrm{CS}} \subseteq \mathscr{D}$.
\end{proposition}
\begin{proof}
In view of \eqref{eq:muvarx}, the constraint $\mathrm{var}_x(T) \preceq \Sigma_\f$ can be written as the following positive semidefinite constraint:
\begin{align}\label{eq:covacons}
\cV(\bm{\cK}) \in \mathbb{S}^{+}_n,~~~~ \cV(\bm{\cK}) := \Sigma_\f - \mathbf{P}_{T+1} \mathfrak{h}(\bm{\cK}) \mathbf{P}_{T+1}\t.
\end{align}
After plugging the expression of $\mathfrak{h}(\bm{\cK})$, which is given in \eqref{eq:huK}, into \eqref{eq:covacons}, it follows that
\begin{equation}
 \cV(\bm{\cK}) = \tilde{\Sigma}_\f - \bm{\zeta}(\bm{\cK})\bm{\zeta}(\bm{\cK})\t,
\end{equation}
where 
\begin{subequations}
\begin{align}
\tilde{\Sigma}_\f & := \Sigma_\f -\mathbf{P}_{T+1} \mathbf{G}_{0} \Sigma_0 \mathbf{G}_{0}\t \mathbf{P}_{T+1}\t, \\ 
\bm{\zeta}(\bm{\cK}) & := \mathbf{P}_{T+1}(\mathbf{G}_{\bm{u}} \bm{\mathcal{K}} + \mathbf{G}_{\bm{w}}) \mathbf{W}^{1/2}.
\end{align} 
\end{subequations}
Thus, the constraint $\cV(\bm{\cK}) \in \mathbb{S}^{+}_n$ can be written as follows:
\begin{equation}
\mathbf{\cN}(\bm{\cK})\in \mathbb{S}_{2n}^+,~~~\mathbf{\cN}(\bm{\cK}) := \begin{bmatrix} \tilde{\Sigma}_\f & \bm{\zeta}(\bm{\cK}) \\ \bm{\zeta}(\bm{\cK})\t & I_n \end{bmatrix}.
\end{equation}
The equivalence of the constraints $\cV(\bm{\cK}) \in \mathbb{S}^{+}_n$ and $\mathbf{\cN}(\bm{\cK}) \in \mathbb{S}_{2n}^+$ follows from the fact that $(\Sigma_\f - \mathbf{G}_{0} \Sigma_0 \mathbf{G}_{0}\t) \in \mathbb{S}_n$ and the (obvious) fact that $I_n \in \mathbb{S}^{++}_n$ (see, for instance, Theorem~4.9 in \cite{b:calafiore2014}; note that $\cV(\bm{\cK})$ is the Schur complement of $I_n$ in $\mathbf{\cN}(\bm{\cK})$). Given that the function $\bm{\zeta}(\bm{\cK})$ is affine, we conclude that the positive semidefinite constraint $\mathbf{\cN}(\bm{\cK})\in \mathbb{S}_{2n}^+$ corresponds to an LMI (convex) constraint and thus the set $\mathscr{D}_3 := \{ (\bar{\bm{u}},\bm{\mathcal{K}}) \in \mathscr{D}: \mathbf{\cN}(\bm{\cK})\in \mathbb{S}_{2n}^+\}$ is convex. Finally, the terminal constraint $\mu_x(T) = \mu_\f$ determines the convex set $\mathscr{D}_2 := \{ (\bar{\bm{u}}, \bm{\mathcal{K}}) \in \mathscr{D}: \cM(\bar{\bm{u}})=0\}$, where $\cM(\bar{\bm{u}})$ is an affine function which is defined as in the proof of Proposition~\ref{prop:constr}. Therefore, the set $\mathscr{D}_{\mathrm{CS}} := \mathscr{D}_2 \cap \mathscr{D}_3$. This completes the proof.
\end{proof}

\section{Control policy parameterization based on truncated histories of disturbances}

If the time horizon $[0,T]_d$ in the formulation of Problem~1 or Problem~2  is large, then the dimensions of the convex programs corresponding to these problems can be significantly large; consequently, the computational cost for their solution can be substantial. To alleviate this issue, one can consider control policies which are sequences of control laws which can be expressed as affine functions of truncated histories of disturbances. 

Specifically, let $\eta \in [1, T-2]_d$ be the length of the subsequence of past disturbances which can be used for the characterization of the control law at time $t \in [0,T-1]_d$, which will now be an affine function of the elements of $W_{\sigma:t-1}$, where $\sigma := \max\{0, t - \eta \}$. We denote this control law by $\kappa_{\mathrm{tr}}(t,W_{\sigma:t-1})$, where
\begin{align}\label{eq:policytr}
\kappa_{\mathrm{tr}}(t,W_{\sigma:t-1})& = \begin{cases} \bar{u}(t) + \sum_{\tau =\sigma}^{t-1}
K_{(t-1,\tau)} w(\tau),\\
~\qquad~\qquad~~\mathrm{if}~ t \in[1,T-1]_d,\\
\bar{u}(0),~\qquad~~\mathrm{if}~ t=0.
\end{cases}
\end{align}
Note that when $\sigma = 0$, then $\kappa_{\mathrm{tr}}(t,W_{\sigma:t-1})$ corresponds to the untruncated control policy $\kappa(t,W_{0:t-1})$ defined in \eqref{eq:policy}.
It is implied that the gains $K_{(t-1,\tau)}$ which do not appear in the expression of the control law given in \eqref{eq:policytr} because the corresponding disturbances have been truncated are set equal to zero. In particular, the vector of inputs $\bm{u}$, can be expressed as follows: 
\begin{equation}\label{eq:controlcmptrun}
\bm{u} = \bar{\bm{u}} + \bm{\mathcal{K}}_{\mathrm{tr}} \bm{w},
\end{equation}
where $\bar{\bm{u}} := \mathrm{vertcat}(\overline{\mathscr{U}})$,
and $\bm{\mathcal{K}}_{\mathrm{tr}} := \big[ \begin{smallmatrix} \mathbf{0} & \mathbf{0} \\ \mathbf{K}_{\mathrm{tr}} & \mathbf{0}\end{smallmatrix} \big]$ with
\begin{align}\label{eq:mathcalKtrunc}
\mathbf{K}_{\mathrm{tr}} & := \left[\begin{smallmatrix} 
K_{(\sigma, \sigma)} & \mathbf{0}  &  \dots & \mathbf{0}  \\
K_{(\sigma+1, \sigma)} & K_{(\sigma+1, \sigma+1)}  &  \dots & \mathbf{0}  \\
\vdots & \vdots & \ddots & \vdots \\
K_{(T-2, \sigma)} & K_{(T-2,\sigma + 1)} & \dots & K_{(T-2, T-2)}
\end{smallmatrix} \right],
\end{align} 
for $t-1 \geq \eta$, and $\mathbf{K}_{\mathrm{tr}} = \mathbf{K}$, where $\mathbf{K}$ is defined as in \eqref{eq:mathcalK}, otherwise. Note that $t-1 \geq \eta$ means that the length of $[0,t-1]_d$ is at least as long as the truncation length $\eta$ (if this is not the case, there in nothing to truncate). All the equations we have derived in Sections~\ref{s:prelim} hold true after substituting $\mathbf{K}$ with $\mathbf{K}_{\mathrm{tr}}$.

\section{Concluding Remarks}\label{s:concl}
In this paper, we have addressed minimum variance and covariance control problems for discrete-time linear systems by means of convex optimization approaches. In our proposed approach, we have utilized a stochastic version of the affine disturbance feedback control parameterization and we have shown that this particular parametrization allows one to reduce the stochastic optimal control problems into tractable convex programs more  directly that other control policy parametrizations such as the state feedback parametrization. We have also shown that by utilizing a variation of the proposed parameterization which utilizes truncations of the histories of disturbances, one can design suboptimal controllers in a way that strikes the desired balance between optimality and computational cost. In our future work, we will consider minimum variance and covariance steering problems with incomplete and imperfect state information. We will also consider the case of nonlinear stochastic systems which we plan to address by means of stochastic model predictive control algorithms.

\section*{Acknowledgments}
This research  has been supported in part by NSF award CMMI-1937957.

\bibliographystyle{ieeetr}
\bibliography{acc_bib2020}
\end{document}